\documentclass[a4paper,dvips,10pt,oneside]{article}
\usepackage[makeroom]{cancel}
\usepackage{graphicx}
\usepackage{epstopdf}
\usepackage{subfig}
\usepackage{color}
\usepackage{float}
\usepackage{leftidx}
\usepackage{bigints}
\usepackage{amssymb,amsmath,dsfont,url,epsfig}
\usepackage[left=1.5in, right=0.91in, top=1.1in, bottom=1.2in, includefoot, headheight=13.6pt]{geometry}
\usepackage[T1]{fontenc}
\usepackage{titlesec, blindtext, color}
\definecolor{gray75}{gray}{0.75}

\newcommand{\sln}{\linespread{1}}
\newcommand*{\email}[1]{\href{mailto:#1}{\nolinkurl{#1}} } 
\titleformat{\chapter}[block]{\LARGE\bfseries\sln}{Chapter \thechapter}{11pt}{\newline\huge\bfseries}
\newtheorem{theorem}{Theorem}[section]
\newtheorem{remark}{Remark}[section]
\newtheorem{definition}{Definition}[section]

\newenvironment{proof}{\paragraph{Proof:}}{\hfill$\square$}

\newtheorem{proposition}{Proposition}[section]

\begin{document}
\title{ Minimal surfaces in three-dimensional Matsumoto space}
\author{Ranadip Gangopadhyay$^*$ and Bankteshwar Tiwari$^*$\\
$^*$DST-CIMS, Institute of Science, Banaras Hindu University, Varanasi-221005, India}
\maketitle

\begin{abstract}
In this paper we consider the Matsumoto metric $F=\frac{\alpha^2}{\alpha-\beta}$, on the three dimensional real vector space and obtain the partial differential equations that characterize the minimal surfaces which are graphs of smooth functions and then we prove that plane is the only such surface. We also obtain the partial differential equation that characterizes the minimal translation surfaces and show that again plane is the only such surface.
\end{abstract}

\section{Introduction}
A minimal surface is a surface that locally minimizes its area. The theory of minimal surfaces arises due to the work of Lagrange in $1762$ when he was trying to find the surface $f=f(x,y)$ having the least area but bounded by a given  closed curve. He could not succeed in finding any solution other than the plane. In $1776$, Meusnier discovered that the helicoid and catenoid are also the solutions. He also showed that this is equivalent to the vanishing of the
mean curvature, and the study of the differential geometry of these surfaces was started.

The study of minimal surfaces in  Riemannian manifolds has been extensively developed \cite{RS}. Many of the developed techniques  have played key roles in geometry and partial differential equations. Examples include monotonicity and tangent cone analysis originating in the regularity theory 
for minimal surfaces, estimates for nonlinear equations based on the maximum principle arising in Bernstein's classical work, and even Lebesgue's 
definition of the integral that he developed in his thesis on the Plateau problem for minimal surfaces \cite{Ra1}.  However, minimal surfaces in Finsler spaces have not been studied and developped at the same pace. The fundamental contribution to the minimal surfaces of Finsler geometry was given by Shen \cite{ZS}. He introduced the notion of mean curvature for immersions into Finsler manifolds and he established some of its properties. As in the Riemannian case, if the mean curvature is identically zero, then the immersion is said to be minimal.
\par The Randers metric is the simplest class of non-Riemannian Finsler metric which is defined  as $F=\alpha+\beta$, where $\alpha$ is a Riemannian metric and $\beta$ is a one-form. M. Souza and K. Tenenblat studied the rotational surfaces to become a minimal surfaces in Minkowski space with Randers metric \cite{MK} and Souza et. al obtained a Bernstein type theorem on a Randers space \cite{MJK}. After that few other authors studied the minimal surfaces on Randers spaces \cite{NC1, NC2, NC3, RK}. V. Balan studied the rotational surface and graph of a smooth function to become a minimal surfaces in Minkowski space with Kropina metric \cite{VB}. N. Cui and Y.B. Shen studied a special class of $(\alpha, \beta)$- metric which satisfies the system of differential equation \cite{NC4}
\begin{equation}
(\phi-s\phi')^{n-1}=1+p(s)+s^2q(s)
\end{equation}
\begin{equation}
(\phi-s\phi')^{n-2}\phi''=q(s)
\end{equation}
where, $p(s)$ and $q(s)$ are arbitrary odd smooth functions. But again Randers metric is the only metric they have found that satifies the above differential equation.\\
Matsumoto slope metric is another class of interesting $(\alpha,\beta)$ metric investigated by M. Matsumoto on the motivation of a letter written by  P. Finsler himself in 1969 to Matsumoto. He considered the following problem:
A person is walking on a horizontal plane with some velocity, and the gravity is acting perpendicularly on this plane. Now suppose the person walks with same velocity on an inclined plane to the horizontal sea level. Now the question is under the presence of gravitational forces, what should be the trajectory the person should walk in the center to reach a given destination in the shortest time?
Based on this, he has formulated the Slope principle \cite{MM1,MM2}. Matsumoto showed that for a hiker walking the slope of a mountain under the presence of gravity, the most efficient time minimizing paths are not the Riemannian geodesics, but the geodesics of the slope metric $F =\frac{\alpha^2}{\alpha-\beta} $.
\par In this paper, we study the minimal surface of graph of a smooth function and translation surface in Minkowski Matsumoto slope metric and prove that plane is the only minimal surface in both the cases.

\section{Preliminaries}
Let $ M $ be an $n$-dimensional smooth manifold. $T_{x}M$ denotes the tangent space of $M$
 at $x$. The tangent bundle of $ M $ is the disjoint union of tangent spaces $ TM:= \sqcup _{x \in M}T_xM $. We denote the elements of $TM$ by $(x,y)$ where $y\in T_{x}M $ and $TM_0:=TM \setminus\left\lbrace 0\right\rbrace $. \\
 \begin{definition}
 \cite{SSZ} A Finsler metric on $M$ is a function $F:TM \to [0,\infty)$ satisfying the following condition:
 \\(i) $F$ is smooth on $TM_{0}$, 
 \\(ii) $F$ is a positively 1-homogeneous on the fibers of tangent bundle $TM$,
 \\(iii) The Hessian of $\frac{F^2}{2}$ with element $g_{ij}=\frac{1}{2}\frac{\partial ^2F^2}{\partial y^i \partial y^j}$ is positive definite on $TM_0$.\\
 The pair $(M,F)$ is called a Finsler space and $g_{ij}$ is called the fundamental tensor.
 \end{definition} 
 A Matsumoto metric on $M$ is a Finsler structure $F$ on $TM$ is given by $F=\frac{\alpha^2}{\alpha-\beta}$, where $\alpha=\sqrt{a_{ij}y^iy^j}$ is a Riemannian metric and  $\beta=b_iy^i$ is a one-form with $0<b<1/2$.
 
 Let $(M^n,F)$ be a $n$-dimensional Finsler manifold. Then the Busemann-Hausdorff volume form is defined as $dV_{BH}=\sigma_{BH}(x)dx$, where
    \begin{equation}\label{eqn9}
    \sigma_{BH}(x)=\frac{Vol(B^n(1))}{Vol\left\lbrace (y^i)\in T_xM : F(x,y)< 1 \right\rbrace },
    \end{equation}
$B^n$ is the Euclidean unit ball in $\mathbb{R}^n$ and $vol$ is the Euclidean volume.\\
Let $( \tilde{M}^m, \tilde{F})$ be a Finsler manifold, with local coordinates $(\tilde{x}^1, \dots, \tilde{x}^m)$ and 
$\varphi : M^n \to (\tilde{M}^m, \tilde{F})$ be an immersion. Then $\tilde{F}$ induces a Finsler metric
on $M$, defined by
\begin{equation}\label{eqn2.1}
F(x,y)=\left( \varphi^*\tilde{F}\right) (x,y)=\tilde{F}\left( \varphi(x),\varphi_*(y)\right) ,\quad \forall (x,y)\in TM.
\end{equation}
At first we assume the following convention: the greek letters $\epsilon, \eta, \gamma, \tau, \dots$ are the indices ranging from $1$ to $n$ and the latin letters $i,j,k,l,\dots$ are the indices ranging from $1$ to $n+1$.
\par  A Minkowski space is a vector space $V^n$ equipped with a Minkowski norm $F$ whose indicatrix is strongly convex. Equivalently, we can say that $F(x, y)$ depends only on $y \in T_x(V^n)$. In this paper we will consider the hypersurface $M^n$ in the Minkowski Matsumoto space $V^{n+1}$ given by the immersion $\varphi : M^n \to (V^{n+1}, F_b)$, where
 $F_b= \frac{\alpha^2}{\alpha-\beta}$, $\alpha$ is the Euclidean metric,
and $\beta$ is a one-form with norm $b$, $0 \le b < 1/2$. Without loss of generality we will
consider $\beta = b dx_{n+1}$. If $M^n$ has local coordinates $x = (x^{\epsilon}), \epsilon= 1,... , n$, and
$\varphi(x) =\left(  \varphi^i(x^{\epsilon})\right)\in V $, $i = 1,\dots, n + 1$, we define
\begin{equation}\label{eqn2.01}
\mathcal{F}(x,z)=\frac{vol (B^n)}{vol (D^n_x)}, \quad z=\left( z^i_{\epsilon}\right)=\frac{\partial \varphi^i}{\partial x^{\epsilon}},
\end{equation}
where,
\begin{equation}
D^n_x=\left\lbrace (y^1,y^2,...,y^n)\in \mathbb{R}^n:F(x,y)<1\right\rbrace.
\end{equation}
The Euclidean volume of $D^n_x$ is given by
\begin{equation}
vol D^n_x=\frac{vol B^n}{\left(1-b^2A^{\epsilon \eta}z^{n+1}_{\epsilon}z^{n+1}_\eta \right)^{\frac{n+1}{2}}\sqrt{det A} }
\end{equation}
where, 
\begin{equation}
A=\left( A_{\epsilon \eta}\right)=\left( \sum\limits_{i=1}^{n+1}z^i_{\epsilon }z^i_{\eta}\right), \quad \textnormal{and} \quad \left( A^{\epsilon \eta}\right) =\left( A_{\epsilon \eta}\right) ^{-1}
\end{equation}
Then the volume form $dV_{BH}$ is given by
\begin{equation}
dV_{BH}=\left(1-b^2A^{\epsilon \eta}z^{n+1}_{\epsilon}z^{n+1}_{\eta} \right)^{\frac{n+1}{2}}\sqrt{det A}dx^1...dx^n.
\end{equation}
 The mean curvature $\mathcal{H}_{\varphi}$, introduced by Z.Shen \cite{ZS} and is given by
 \begin{equation}
 \mathcal{H}_{\varphi}(v)=\frac{1}{\mathcal{F}}\left\lbrace \frac{\partial^2 \mathcal{F}}{\partial z^i_{\epsilon}\partial z^j_{\eta}} \frac{\partial^2 \mathcal{\varphi}^j}{\partial x^{\epsilon}\partial x^{\eta}}+\frac{\partial^2 \mathcal{F}}{\partial z^i_{\epsilon}\partial \tilde{x}^j} \frac{\partial \mathcal{\varphi}^j}{\partial x^{\epsilon}} -\frac{\partial \mathcal{F}}{\partial \tilde{x}^i}\right\rbrace v^i.
 \end{equation}
 Here $v=(v^i)$ is a vector field over $\tilde{M}$ and $\mathcal{H}_{\varphi}(v)$depends linearly on $v$. Also the mean curvature vanishes on $\varphi_*(TM)$. Whenever $(V ,F)$ is a Minkowski space, the expression of the mean curvature reduces to
 \begin{equation}
 \mathcal{H}_{\varphi}(v)=\frac{1}{\mathcal{F}}\left\lbrace \frac{\partial^2 \mathcal{F}}{\partial z^i_{\epsilon}\partial z^j_{\eta}} \frac{\partial^2 \mathcal{\varphi}}{\partial x^{\epsilon}\partial x^{\eta}}\right\rbrace v^i.
 \end{equation}
 The immersion $\varphi$ is said to be minimal when $\mathcal{H}_{\varphi}=0$.
 
\section{The partial differential equation of minimal surfaces in Matsumoto spaces}
In this section we obtain the volume form of Matsumoto metric and with the help of that for any immersion $\varphi:M^2\to (V^3,F_b)$ we obtain the characteristic differential equation for which $\varphi$ is minimal.
 \begin{proposition}\label{prop1}\cite{XZ}
 Let $F=\alpha\phi(s)$, $s=\beta/\alpha$, be an $(\alpha,\beta)$-metric on an $n$-dimensional manifold $M$. Let
 \begin{equation}\label{eqn3.1}
 f(b) := \begin{cases} \frac{\int\limits_{0}^{\pi}\sin^{n-2}(t)dt}{\int\limits_{0}^{\pi}\frac{\sin^{n-2}(t)}{\phi(b\cos (t))^n}dt} &\mbox{if } dV=dV_{BH} \\
 \frac{\int\limits_{0}^{\pi}\sin^{n-2}(t)T(b\cos (t))dt}{\int\limits_{0}^{\pi}\sin^{n-2}(t)dt} & \mbox{if } dV=dV_{HT} \end{cases}
 \end{equation}
 Then the volume form $dV$ is given by $$dV=f(b)dV_{\alpha}$$
 where, $dV_{\alpha}=\sqrt{det(a_{ij})}dx$ denotes the Riemannian volume form $\alpha$.
 \end{proposition}
 
 \begin{theorem}\label{thm1}
 Let $F=\frac{\alpha^2}{\alpha-\beta}$, be the Matsumoto metric on a $2$-dimensional manifold $M$. Then Bausmann-Hausdorff volume form of $F$ is given by
 $$dV_{BH}=\frac{2}{2+b^2}\sqrt{det(a_{ij})}dx$$
 \end{theorem}
 
 \begin{proof}
 For Matsumoto surface we have $\phi(s)=\frac{1}{1-s}$ and $n=2$. Therefore,  from \eqref{eqn3.1} we have
 \begin{equation}\label{eqn3.01}
 \begin{split}
 f(b) &= \frac{\int\limits_{0}^{\pi}dt}{\int\limits_{0}^{\pi}(1-b\cos t)^2dt} \\ & =\frac{2}{2+b^2}
 \end{split}
 \end{equation}
 Hence, the theorem follows.

  \end{proof}
  \begin{theorem}
  Let $\varphi:M^2 \to (V^3,F_b)$ be an immersion in a Matsumoto space with local coordinates $(\varphi^j(x))$. Then $\varphi$ is minimal if and only if 
   \begin{eqnarray}\label{MCE0}
    \frac{\partial^2 \varphi^j}{\partial x^{\epsilon}\partial x^{\eta}}v^i\left[\frac{2C^2+3E}{(2C^2+E)^2} \frac{\partial^2 C^2}{\partial z^i_{\epsilon}\partial z^j_{\eta}}-\frac{2C^2}{(2C^2+E)^2}\frac{\partial^2 E}{\partial z^i_{\epsilon}\partial z^j_{\eta}}\nonumber \right.\\ \left.- \frac{2(4C^4+12C^2E-12C^3-3E^2)}{(2C^2+E)^3} \frac{\partial C}{\partial z^i_{\epsilon}}\frac{\partial C}{\partial z^j_{\eta}}+\frac{4C^2}{(2C^2+E)^3} \frac{\partial E }{\partial z^i_\epsilon}\frac{\partial E }{\partial z^j_\eta}\nonumber\right.\\ \left.+\frac{4C^3-6CE}{(2C^2+E)^3}\left(\frac{\partial C }{\partial z^i_\epsilon}\frac{\partial E }{\partial z^j_\eta}+\frac{\partial E }{\partial z^i_\epsilon}\frac{\partial C }{\partial z^j_\eta} \right) \right]&=&0 \hspace{1.0cm}
    \end{eqnarray}
  \end{theorem}
  where, $C=\sqrt{det(A)}$ and
     \begin{equation}
     E=b^2\sum\limits_{k=1}^{3}(-1)^{\gamma + \tau}z^{k}_{\bar{\gamma}}z^{k}_{\bar{\tau}}z^{3}_{\gamma}z^{3}_{\tau}
     \end{equation}
  \begin{proof}
   From the discussion in Section $2$ the volume form of a Matsumoto metric can be written as 
   \begin{equation}
   dV_{BH}=\frac{2}{2+b^2A^{\epsilon \eta}z^3_{\epsilon}z^3_{\eta}}\sqrt{det(A)}dx^1dx^2.
   \end{equation}
  Let $B=b^2A^{\epsilon \eta}z^3_{\epsilon }z^3_{\eta}$. Then using \eqref{eqn3.01} in \eqref{eqn2.01}  one can write 
   \begin{equation}\label{eqn3.5}
   \mathcal{F}(x,z)=\frac{2}{2+B}C.
   \end{equation}
   Since $A_{\epsilon\eta}=\sum\limits_{i=1}^{3}z^i_{\epsilon}z^i_{\eta}$, its inverse matrix is given by,
   \begin{equation*}
   A^{\epsilon\eta}= \sum\limits_{i=1}^{3}\frac{(-1)^{\epsilon+\eta}}{det A}z^i_{\bar{\epsilon}}z^i_{\bar{\eta}}
   \end{equation*}
      Here the notation bar for any greek letters ranging from $1$ to $2$ is defined by
      \begin{equation*}
      \bar{\tau}=\delta_{\tau 2}+2\delta_{\tau 1}.
      \end{equation*}
 Hence we get 
   \begin{equation}
   B=b^2 \sum\limits_{i=1}^{3}\frac{(-1)^{\gamma + \tau}}{det A}z^{i}_{\bar{\gamma}}z^{i}_{\bar{\tau}}z^{3}_{\gamma}z^{3}_{\tau}. ,
   \end{equation}
   Hence, we have,
   \begin{equation}
   B=\frac{E}{C^2}
   \end{equation}
   where,
   \begin{equation}
   E=b^2\sum\limits_{k=1}^{3}(-1)^{\gamma + \tau}z^{k}_{\bar{\gamma}}z^{k}_{\bar{\tau}}z^{3}_{\gamma}z^{3}_{\tau}
   \end{equation}
   Therefore, \eqref{eqn3.5} becomes 
   \begin{equation}\label{eqn3.6}
   \mathcal{F}(x,z)=\frac{2C^3}{2C^2+E}.
    \end{equation}
    Differentiating \eqref{eqn3.6} with respect to $z^i_{\epsilon}$ we get
    \begin{equation}\label{eqn3.7}
    \frac{\partial \mathcal{F} }{\partial z^i_{\epsilon}}= \frac{4C^4+6C^2E}{(2C^2+E)^2}\frac{\partial C }{\partial z^i_{\epsilon}} - \frac{2C^3}{(2C^2+E)^2}\frac{\partial E }{\partial z^i_{\epsilon}}.
    \end{equation}
     Differentiating \eqref{eqn3.7} with respect to $z^j_{\eta}$ we get
     \begin{equation}\label{eqn3.8}
     \begin{split}
     \frac{\partial^2 \mathcal{F} }{\partial z^i_{\epsilon}\partial z^j_{\eta}}
     =\frac{4C^4+6C^2E}{(2C^2+E)^2}\frac{\partial^2 C}{\partial z^i_{\epsilon}\partial z^j_{\eta}}-\frac{2C^3}{(2C^2+E)^2}\frac{\partial^2 E}{\partial z^i_{\epsilon}\partial z^j_{\eta}}\hspace{2.0cm} \\+\frac{12CE^2-8C^3E}{(2C^2+E)^3}\frac{\partial C}{\partial z^i_{\epsilon}}\frac{\partial C}{\partial z^j_{\eta}}+\frac{4C^4-6C^2E}{(2C^2+E)^3}\left(\frac{\partial C }{\partial z^i_\epsilon}\frac{\partial E }{\partial z^j_\eta}+\frac{\partial E }{\partial z^i_\epsilon}\frac{\partial C }{\partial z^j_\eta} \right)\\+\frac{4C^3}{(2C^2+E)^3} \frac{\partial E }{\partial z^i_\epsilon}\frac{\partial E }{\partial z^j_\eta}.
     \end{split}
     \end{equation}
     The Matsumoto metric has vanishing mean curvature iff 
     \begin{equation}\label{eqn3.9}
     \frac{\partial^2 \mathcal{F} }{\partial z^i_{\epsilon}\partial z^j_{\eta}}\frac{\partial^2 \varphi^j}{\partial x^{\epsilon}\partial x^{\eta}}v^i=0.
     \end{equation}
     Therefore, using \eqref{eqn3.8} in \eqref{eqn3.9} we obtain \eqref{MCE0}.
  \end{proof}
 
 \section{The characterization of minimal surfaces which are the graph of a function}
 In this section we study the graph of a function $M^2$ in Matsumoto space $(V^3,F_b)$, where $V^3$ is a real vectoe space and $F_b=\frac{{\tilde{\alpha}}^2}{\tilde{\alpha}-\tilde{\beta}}$ is a Matsumoto metric, where $\tilde{\alpha}$ is the Euclidean metric and $\tilde{\beta}=bdx^3$ is a one-form. Here we consider the immersion $\varphi : U \subset \mathbb{R}^2 \to (V^3, F_b)$ given by $\varphi(x^1, x^2) =(x^1, x^2, f(x^1, x^2))$. At first we show that the pullback metric of $F_b$ bt $\varphi$ is again a Matsumoto metric and then find the characterization equation the surface to be minimal.
 \begin{proposition}
 Let $F_b=\frac{{\tilde{\alpha}}^2}{\tilde{\alpha}-\tilde{\beta}}$ is a Matsumoto metric, where $\tilde{\alpha}$ is the Euclidean metric and $\tilde{\beta}=bdx^3$ is a one-form on the real vector space $V^3$. Now suppose $\varphi : U \subset \mathbb{R}^2 \to (V^3, F_b)$ given by $\varphi(x^1, x^2) =(x^1, x^2, f(x^1, x^2))$, where $f$ is a real valued smooth function be an immersion Then the pullback metric on $U$ defined by \eqref{eqn2.1} is again a Matsumoto metric.
 \end{proposition}
 \begin{proof}
 We have,
 \begin{equation}
F_b=\frac{{\tilde{\alpha}}^2}{\tilde{\alpha}-\tilde{\beta}}=\frac{(d\tilde{x}^1)^2+(d\tilde{x}^2)^2+(d\tilde{x}^3)^2}{\sqrt{(d\tilde{x}^1)^2+(d\tilde{x}^2)^2+(d\tilde{x}^3)^2}-bd\tilde{x}^3}
 \end{equation}
 Now,
 \begin{equation}
 \begin{split}
 \varphi^*(d\tilde{x}^1)=dx^1, \quad \varphi^*(d\tilde{x}^2)=dx^2, \\ \varphi^*(d\tilde{x}^3)=(d(f(x^1, x^2)))=f_{x^1}dx^1+f_{x^2}dx^2
 \end{split}
 \end{equation}
 Therefore,
 \begin{equation}
 \varphi^*F_b=\frac{(1+f^2_{x^1})(dx^1)^2+2f_{x^1}f_{x^2}dx^1dx^2+(1+f^2_{x^2})(dx^2)^2}{\sqrt{(1+f^2_{x^1})(dx^1)^2+2f_{x^1}f_{x^2}dx^1dx^2+(1+f^2_{x^2})(dx^2)^2}-b(f_{x^1}dx^1+f_{x^2}dx^2)}
 \end{equation}
 which is a Matsumoto metric of the form $\frac{\alpha^2}{\alpha-\beta}$, where
 \begin{equation*}
 \alpha^2=(1+f^2_{x^1})(dx^1)^2+2f_{x^1}f_{x^2}dx^1dx^2+(1+f^2_{x^2})(dx^2)^2
 \end{equation*}
 is a Riemannian metric and
 \begin{equation*}
 \beta=b(f_{x^1}dx^1+f_{x^2}dx^2)
 \end{equation*}
 is a one-form.
 \end{proof}
 \begin{theorem}\label{thm4.1}
 An immersion $\varphi : U \subset \mathbb{R}^2 \to (V^3, F_b)$ given by $\varphi(x^1, x^2) =(x^1, x^2, f(x^1, x^2))$ is minimal, if and only if, $f$ satisfies
\begin{equation}\label{eqn4.0}
\sum\limits_{ \epsilon, \eta=1,2}\left[ T_b(T_b-2b^2)\left(\delta_{\epsilon \eta}-\frac{f_{x^{\epsilon}}f_{x^{\eta}}}{W^2} \right)+ 4b^2(T_b +4 b^2) \frac{f_{x^{\epsilon}}f_{x^{\eta}}}{W^2} \right] f_{x^{\epsilon}x^{\eta}}= 0,
\end{equation}
 where,
 $W^2 = 1 + f^2_{x^1}+ f^2_{x^2}, \qquad  T_b = 2W^2 + b^2(W^2-1)$.
  \end{theorem}
 \begin{proof}
  The mean curvature vanishes on tangent vectors of the immersion $\varphi$. Therefore, we need to consider a vector field $v$ such that the set $\left\lbrace v, \varphi_{x^1}, \varphi_{x^2} \right\rbrace $ is linearly independent. Therefore, we consider $v=\varphi_{x^1}\times \varphi_{x^2}$. Then $v=\left( v^1,v^2,v^3\right)=\left(-f_{x^1}, -f_{x^2}, 1 \right)$.
 Here we have
 \begin{equation}\label{eqn4.1}
 A = 
 \begin{pmatrix}
 1+f^2_{x^1} & f_{x^1}f_{x^2} \\
 f_{x^1}f_{x^2} & 1+f^2_{x^2} \\
 \end{pmatrix},
 \end{equation}
 \begin{equation}\label{eqn4.2}
 C=\sqrt{det A}=W, \quad E=b^2\left( W^2-1\right), 
 \end{equation}
 
 By some simple calculations we can have
 \begin{equation}\label{eqn4.3}
 \frac{\partial C}{\partial z^i_{\epsilon}}v^i=0,
 \end{equation}
 \begin{equation}\label{eqn4.4}
 \frac{\partial E}{\partial z^i_{\epsilon}}v^i= 2b^2(\delta_{\epsilon 1}f_{x^1}+\delta_{\epsilon 2}f_{x^2}),
 \end{equation}
 \begin{equation}\label{eqn4.5}
 \frac{\partial C}{\partial z^j_{\eta}}\frac{\partial^2\varphi^j} {\partial x^{\epsilon}\partial x^{\eta}}v^i= \frac{f_{x^1}f_{x^{\epsilon}x^1}+f_{x^2}f_{x^{\epsilon}x^2}}{W},
 \end{equation}
 \begin{equation}\label{eqn4.6}
 \frac{\partial E}{\partial z^j_{\eta}}\frac{\partial^2\varphi^j} {\partial x^{\epsilon}\partial x^{\eta}}= 2b^2(f_{x^1}f_{x^{\epsilon}x^1}+f_{x^2}f_{x^{\epsilon}x^2}),
 \end{equation}
 \begin{equation}\label{eqn4.7}
 \frac{\partial^2 E}{\partial z^i_{\epsilon}\partial z^j_{\eta}}\frac{\partial^2\varphi^j} {\partial x^{\epsilon}\partial x^{\eta}}v^i= 2b^2\left[ (1+f^2_{x^2})f_{x^1x^1}-2f_{x^1}f_{x^2}f_{x^1x^2}+(1+f^2_{x^1})f_{x^2x^2}\right], 
 \end{equation}
 \begin{equation}\label{eqn4.8}
 \frac{1}{2}\frac{\partial^2 C^2}{\partial z^i_{\epsilon}\partial z^j_{\eta}}\frac{\partial^2\varphi^j} {\partial x^{\epsilon}\partial x^{\eta}}v^i= \left[ (1+f^2_{x^2})f_{x^1x^1}-2f_{x^1}f_{x^2}f_{x^1x^2}+(1+f^2_{x^1})f_{x^2x^2}\right].
 \end{equation}
 Using \eqref{eqn4.3} in \eqref{MCE0} we have
 \begin{eqnarray}\label{MCE001}
 \frac{\partial^2 \varphi^j}{\partial x^{\epsilon}\partial x^{\eta}}v^i\left[ \frac{\partial^2 C^2}{\partial z^i_{\epsilon}\partial z^j_{\eta}}(2C^2+3E)(2C^2+E)-\frac{\partial^2 E}{\partial z^i_{\epsilon}\partial z^j_{\eta}} 2C^2(2C^2+E)\nonumber \right.\\ \left.+\left\lbrace\frac{\partial E }{\partial z^i_\epsilon}\frac{\partial C }{\partial z^j_\eta} \left(4C^3-6CE \right)+4C^2 \frac{\partial E }{\partial z^i_\epsilon}\frac{\partial E }{\partial z^j_\eta} \right\rbrace  \right]&=&0. \hspace{1.0cm}
 \end{eqnarray}
 Let $T_b=2C^2+E$. Then we have the followings:
 \begin{equation}
 \begin{split}
  T_b=2b^2+b^2(W^2-1), \quad 2C^2+3E=2W^2+3b^2(W^2-1), \\ \left(4C^3-6CE \right)=2W\left\lbrace T_b-4b^2(W^2-1)\right\rbrace.
 \end{split}
 \end{equation}
 Putting all these values in \eqref{MCE001} we get
 \begin{equation}\label{eqn4.my}
 \begin{split}
  T_b(T_b-2b^2)\left[ (1+f^2_{x^2})f_{x^1x^1}-2f_{x^1}f_{x^2}f_{x^1x^2}+(1+f^2_{x^1})f_{x^2x^2}\right]\\+4b^2(T_b+4b^2)\left[ f^2_{x^1}f_{x^1x^1}+2f_{x^1}f_{x^2}f_{x^1x^2}+f^2_{x^2}f_{x^2x^2}\right]=0
 \end{split}
 \end{equation}
 Equation \eqref{eqn4.my} can be written as
 \begin{equation}
 \begin{split}
\left[T_b(T_b-2b^2)\left(W^2-f^2_{x^1} \right)+ 4b^2(T_b+4b^2)f^2_{x^1} \right]  f_{x^1x^1}\\ -2\left[ T_b(T_b-2b^2)-4b^2(T_b+4b^2)\right]f_{x^1}f_{x^2} f_{x^1x^2}\\
+\left[ T_b(T_b-2b^2)\left(W^2-f^2_{x^2} \right)+ 4b^2(T_b+4b^2)f^2_{x^2}\right] f_{x^2x^2} 
 \end{split}
 \end{equation}
 The above equation is equivalent to \eqref{eqn4.0}. Hence, we complete the proof.
 
   \end{proof}
     \begin{theorem}\label{thm2}
    An immersion $\varphi : U \subset \mathbb{R}^2 \to (V^3, F_b)$ given by $\varphi(x^1, x^2) =(x^1, x^2, f(x^1, x^2))$ is minimal, if and only if, $f$ satisfies
   \begin{eqnarray}\label{eqn4.25}
   \sum\limits_{ \epsilon, \eta =1,2}\left[ S_b(S_b-2b^2w^2)\left(\delta_{\epsilon\eta}-\frac{f_{x^{\epsilon}}f_{x^{\eta}}}{W^2} \right)\nonumber \right.\\ \left.+ 4b^2(S_b +4 b^2w^2)\left( k_{\epsilon}+\frac{f_{x^{\epsilon}}}{W^2}\right)\left( k_{\eta}+\frac{f_{x^{\eta}}}{W^2}\right) \right] f_{x^{\epsilon}x^{\eta}}&=& 0
   \end{eqnarray}
    where $k_i$ are real numbers such that $\sum\limits_{i=1}^{3}k^2_i=1$ and
    \begin{equation}
    W^2 = 1 + f^2_{x^1}+ f^2_{x^2},\quad  S_b = b^2 + (2+b^2)W^2, \quad w=-k_1f_{x^1}-k_2f_{x^2}+k_3.
    \end{equation}
     \end{theorem}
       \begin{proof}
       The proof of this theorem is similar to the previous theorem. Let us consider the  immersion $\varphi$ is a graph of a function over an open
       subset of a plane of $V^3$. Then $\varphi$ can be written in the form
       \begin{equation}
       \varphi(x^1,x^2)=\left( x^1,x^2,f(x^1,x^2)\right)\left( m_{ij} \right), 
       \end{equation}
       where $(m_{ij} )$ is a $3 \times 3$ orthogonal matrix, $(x^1,x^2)\in  U \subset \mathbb{R}^2$ and the surface is a graph over the plane $m_{31}x + m_{32}y + m_{33}z = 0$.\\
       We now consider the vector field $v = (v^1, v^2, v^3)$ which is linearly independent with $\varphi_{x^1}$ and $\varphi_{x^2}$. Hence we consider $v=\varphi_{x^1}\times \varphi_{x^2}$. Therefore,
       \begin{equation*}
       v^i = -f_{x^1}m_{1i} -f_{x^2}m_{2i} + m_{3i} ,
       \end{equation*}
       Now note that
       \begin{equation}
       z^i_{\eta}= \frac{\partial \varphi^i}{\partial x^{\eta}}=m_{\eta i}+f_{x^{\eta}}m_{3i}, \quad \frac{\partial^2 \varphi^i}{\partial x_{\epsilon}\partial_{x^{\eta}}}=f_{x^{\epsilon}{x^{\eta}}}m_{3i}.
       \end{equation}
Further, for all $i=1,2,3$ and $\eta,\gamma,\epsilon=1,2$, we have,
\begin{equation}
\sum\limits_{i=1}^{3}z^i_{\eta}v^i=0,\quad \sum\limits_{i=1}^{3}v^im_{3i}=1, \quad \sum\limits_{i=1}^{3}z^i_{\eta}m_{3i}=f_{x^{\eta}}, \quad \sum\limits_{i=1}^{3}z^i_{\gamma}\frac{\partial^2 \varphi^i}{\partial x^{\epsilon}\partial {x^{\eta}}}=f_{x^{\gamma}}f_{x^{\epsilon}{x^{\eta}}}. 
\end{equation}
Here the values of $A$ and $C$ are as given in \eqref{eqn4.1} and \eqref{eqn4.2} respectively. And $E=b^2(W^2-w^2)$, $w=v^3$. Let $m_{3i}=k_i$. Therefore, as obtained in Theorem \ref{thm4.1} similarly we obtain the followings:
 \begin{equation}\label{eqn4.10}
 \frac{\partial C}{\partial z^i_{\epsilon}}v^i=0,
 \end{equation}
 \begin{equation}\label{eqn4.11}
 \frac{\partial E}{\partial z^i_{\epsilon}}v^i= 2b^2\left( z^3_{\epsilon}A_{\bar{\epsilon}\bar{\epsilon}}-z^3_{\bar{\epsilon}}A_{\epsilon\bar{\epsilon}}\right) w, \quad \forall \epsilon
 \end{equation}
 \begin{equation}\label{eqn4.12}
 \frac{\partial C}{\partial z^j_{\eta}}\frac{\partial^2\varphi^j} {\partial x^{\epsilon}\partial x^{\eta}}v^i= \frac{f_{x^1}f_{x^{\epsilon}x^1}+f_{x^2}f_{x^{\epsilon}x^2}}{W}, \quad \forall \epsilon
 \end{equation}
 \begin{equation}\label{eqn4.13}
 \frac{\partial E}{\partial z^j_{\eta}}\frac{\partial^2\varphi^j} {\partial x^{\epsilon}\partial x^{\eta}}= 2b^2\left[  \left( f_{x^1}+k_1w\right) f_{x^{\epsilon}x^1}+\left( f_{x^2}+k_2w\right)f_{x^{\epsilon}x^2}\right] , \quad \forall \epsilon
 \end{equation}
 \begin{eqnarray}\label{eqn4.14}
 \frac{\partial^2 E}{\partial z^i_{\epsilon}\partial z^j_{\eta}}\frac{\partial^2\varphi^j} {\partial x^{\epsilon}\partial x^{\eta}}v^i= 2b^2\left[\left\lbrace 1+f^2_{x^2}-k_1\left( k_1W^2+f_{x^1}w\right) \right\rbrace f_{x^1x^1}\nonumber\right.\\ \left.-\left\lbrace \left(1+k^2_3 \right)f_{x^1}f_{x^2}+ k_1k_2W^2+k_1k_3f_{x^2}+k_2k_3f_{x^1}+k_1k_2\right\rbrace f_{x^1x^2}\nonumber\right.\\ \left.+ \left\lbrace 1+f^2_{x^1}-k_2\left( k_2W^2+f_{x^2}w\right) \right\rbrace f_{x^2x^2}\right], 
 \end{eqnarray}
 \begin{equation}\label{eqn4.15}
 \frac{1}{2}\frac{\partial^2 C^2}{\partial z^i_{\epsilon}\partial z^j_{\eta}}\frac{\partial^2\varphi^j} {\partial x^{\epsilon}\partial x^{\eta}}v^i= \left[ (1+f^2_{x^1})f_{x^2x^2}-2f_{x^1}f_{x^2}f_{x^1x^2}+(1+f^2_{x^2})f_{x^1x^1}\right]. 
 \end{equation}
 Using \eqref{eqn4.10} in \eqref{MCE0} we have
 \begin{eqnarray}\label{MCE003}
 \frac{\partial^2 \varphi^j}{\partial x^{\epsilon}\partial x^{\eta}}v^i\left[ \frac{\partial^2 C^2}{\partial z^i_{\epsilon}\partial z^j_{\eta}}(2C^2+3E)(2C^2+E)-\frac{\partial^2 E}{\partial z^i_{\epsilon}\partial z^j_{\eta}} 2C^2(2C^2+E)\nonumber \right.\\ \left.+\left\lbrace\frac{\partial E }{\partial z^i_\epsilon}\frac{\partial C }{\partial z^j_\eta} \left(4C^3-6CE \right)+4C^2 \frac{\partial E }{\partial z^i_\epsilon}\frac{\partial E }{\partial z^j_\eta} \right\rbrace  \right]&=&0 \hspace{1.0cm}
 \end{eqnarray}
  Let $S_b=2C^2+E$. Then, 
  \begin{equation}
  \begin{split}
   S_b=2b^2+b^2(W^2-w^2), \quad 2C^2+3E=2W^2+3b^2(W^2-w^2), \\ \left(4C^3-6CE \right)=2W\left\lbrace S_b-4b^2(W^2-w^2)\right\rbrace.
  \end{split}
  \end{equation}
  Putting all these values in \eqref{MCE003} we get \eqref{eqn4.25}.
        \end{proof}
        \begin{remark}
        Observe that when $k_1 = k_2 = 0$ and $k_3 = 1$, then equation \eqref{MCE001} reduces to \eqref{eqn4.0}. 
        \end{remark}
\begin{definition}\cite{LS1}
A differential equation is said to be a elliptic equation of mean curvature type on a domain $\Omega\subset \mathbb{R}^2$ if 
\begin{equation}
\sum\limits_{ \epsilon,\eta=1,2}a_{\epsilon\eta}(x,f,\nabla f)f_{x^{\epsilon}x^{\eta}}=0
\end{equation}
where $a_{\epsilon\eta}, \epsilon,\eta = 1, 2$ are given real-valued functions on $\Omega \times \mathbb{R} \times \mathbb{R}^2$, $x\in \Omega$, $f:\Omega \to \mathbb{R}$  with
\begin{equation}
|\xi|^2- \frac{(p\cdot \xi)^2}{1+|p|^2}\sum\limits_{ \epsilon,\eta=1,2}a_{\epsilon\eta}(x,u,p)\xi_{\epsilon}\xi_{\eta}\le\left(1+\mathcal{C} \right) \left[|\xi|^2- \frac{(p\cdot\xi)^2}{1+|p|^2}\right] 
\end{equation}
for all $u \in \mathbb{R}$, $p \in \mathbb{R}^2$ and $\xi \in \mathbb{R}^2\setminus \left\lbrace 0 \right\rbrace $.
 \end{definition}
\begin{theorem}\label{thm3}
Let $\varphi : U \subset \mathbb{R}^2 \to (V^3, F_b)$ be an immersion which is the graph of a function $f (x^1, x^2)$ over a plane. Then $\varphi$ is minimal, if and only if, $f$ satisfies the elliptic differential equation, of mean curvature type, given by
\begin{equation}\label{eqn4.34}
\sum\limits_{ \epsilon,\eta=1,2}a_{\epsilon\eta}(x,f,\nabla f)f_{x^{\epsilon}x^{\eta}}=0
\end{equation}
where,
\begin{equation}
a_{\epsilon\eta}=\delta_{\epsilon\eta}-\frac{f_{x^{\epsilon}}f_{x^{\eta}}}{W^2}+R_bW^2+\left( k_{\epsilon}+\frac{f_{x^{\epsilon}}}{W^2}\right)\left( k_{\eta}+\frac{f_{x^{\eta}}}{W^2}\right),
\end{equation}
\begin{equation}
R_b=\frac{4b^2(S_b +4 b^2w^2)}{S_b(S_b-2b^2w^2)},
\end{equation}
\end{theorem}
 \begin{proof}
In Theorem \ref{thm2}, we already prove that $\varphi$ is minimal if and only if it satisfies \eqref{eqn4.25}. Since for a Matsumoto metric $0<b<1/2$, therefore, we have from the definition, $S_b>0$. And also
\begin{equation}\label{eqn4.35}
(S_b-2b^2w^2)=b^2+(2+b^2)W^2-2b^2w^2=b^2+(2-b^2)W^2+2b^2(W^2-w^2)
\end{equation}
Now,
\begin{equation}\label{eqn4.36}
W^2-w^2=(k_2f_{x^1}-k_1f_{x^2})^2+(k_1+k_3f_{x^1})^2+(k_2+k_3f_{x^2})^2>0.
\end{equation}
Since, $0<b<1/2$, using \eqref{eqn4.36} in \eqref{eqn4.35}, we have, $(S_b-2b^2w^2)>0$.\\
Now dividing both sides of \eqref{eqn4.25} by $S_b(S_b-2b^2w^2)$, we get \eqref{eqn4.34}.\\
Let us consider $\xi\in \mathbb{R}^2\setminus \left\lbrace 0 \right\rbrace $, $x,t\in \mathbb{R}^2$ and $u\in \mathbb{R}$ and we define
\begin{equation}
h_{\epsilon\eta}(u)= \delta_{\epsilon\eta}-\frac{t_{\epsilon}t_{\eta}}{W^2(u)}.
\end{equation}
Hence, we have,
\begin{equation}\label{eqn5.36}
\sum\limits_{\epsilon,\eta=1}^{2}h_{\epsilon\eta}(t)\xi_i\xi_j=\frac{|\xi|^2}{W^2}(1+|t|^2\sin^2 \theta),
\end{equation}
where, $\theta$ is the angle function between $t$ and $\xi$. We also have from 
\begin{equation}
\sum\limits_{\epsilon,\eta=1}^{2}a_{\epsilon\eta}(x,u,t)\xi_{\epsilon}\xi_{\eta}=\sum\limits_{\epsilon\eta=1}^{2}h_{\epsilon\eta}(t)\xi_{\epsilon}\xi_{\eta}+R_bW^2\left[(k_1,k_2)\cdot\xi+\frac{w}{W^2}t\cdot\xi \right]^2,
\end{equation}
where $\cdot$ represents the Euclidean inner product.\\
Since $R_b>0$, for all $\xi\in \mathbb{R}^2\setminus \left\lbrace 0 \right\rbrace $, from \eqref{eqn5.36}  we have,
\begin{equation}\label{eqn4.37}
\sum\limits_{\epsilon,\eta=1}^{2}a_{\epsilon\eta}(x,u,t)\xi_{\epsilon}\xi_{\eta}\ge\sum\limits_{\epsilon,\eta=1}^{2}h_{\epsilon\eta}(t)\xi_{\epsilon}\xi_{\eta}\ge\frac{|\xi|^2}{W^2}>0.
\end{equation}
Hence, \eqref{eqn4.34} is an elliptic equation.
Now we prove that it is a differential equation of mean curvature type for which we need to show that there exists a constant $\mathcal{C}$ such that, for all
\begin{equation}
\sum\limits_{\epsilon, \eta=1}^{2}h_{\epsilon \eta}(x,u,t)\xi_{\epsilon}\xi_{\eta} \le \sum\limits_{\epsilon, \eta=1}^{2}a_{\epsilon  \eta}(x,u,t)\xi_{\epsilon}\xi_{\eta}\le(1+\mathcal{C})\sum\limits_{\epsilon, \eta=1}^{2}h_{\epsilon \eta}(x,u,t)\xi_{\epsilon}\xi_{\eta}.
\end{equation}
The first inequality is immediate from \eqref{eqn4.37}. To prove the second inequality we need to show that 
\begin{equation}
R_bW^2\left[(k_1,k_2).\xi+\frac{w}{W^2}t.\xi \right]^2\le\mathcal{C}\sum\limits_{\epsilon,\eta=1}^{2}h_{\epsilon\eta}(x,u,t)\xi_{\epsilon}\xi_{\eta},
\end{equation}
where, $w=-k_1t_1-k_2t_2+k_3$.\\
From  \eqref{eqn5.36} we have, 
\begin{equation}
W^2\left[(k_1,k_2).\xi+\frac{w}{W^2}t.\xi \right]^2=\frac{\left[ W^2|(k_1,k_2)|\cos \gamma+w|t|\cos \theta\right] ^2}{1+|t|^2\sin^2\theta}\sum\limits_{\epsilon, \eta=1}^{2}h_{\epsilon \eta}(x,u,t)\xi_{\epsilon}\xi_{\eta},
\end{equation}
where $\gamma$ is the angle between $(k_1,k_2)$ and $\xi$. Hence, we need to show that
\begin{equation}\label{eqn4.39}
R_b\frac{\left[ W^2|(k_1,k_2)|\cos \gamma+w|t|\cos \theta\right] ^2}{1+|t|^2\sin^2\theta}\le\mathcal{C}.
\end{equation}
It can be seen that $W^2\ge 1$. When $W^2 = 1$, then, $t = 0$. In that case, we have
\begin{equation*}
0\le R_b\left[ |(k_1,k_2)|\cos \gamma\right] ^2\le R_b(0)(k_1^2+k_2^2).
\end{equation*}
Therefore, taking $\mathcal{C}=R_b(0)(k_1^2+k_2^2)$ we prove the inequality.\\
Now suppose  $W^2 > 1$ and $\sin \theta = 0$. in that case $t \ne 0$ and the vectors $t$ and $\xi$ are parallel to each other. Hence,
\begin{equation}\label{eqn4.40}
\left[ W^2 |(k_1, k_2)| \cos \gamma + w|t| \cos \theta\right]^2  = \left[ |(k_1, k_2)| \cos \gamma + k_3|t| cos \theta\right]^2.
\end{equation}
Equation \eqref{eqn4.40} implies that $R_b\frac{\left[ W^2|(k_1,k_2)|\cos \gamma+w|t|\cos \theta\right] ^2}{1+|t|^2\sin^2\theta}$ is a rational function of $|t|$ whose numerator is of degree less than or equal to $4$, and denominator is of degree $4$ and hence it is a bounded function as $|t|$ (or, equivalently $W$) tends to infinity.\\
Now suppose $W^2 > 1$ and $\sin \theta \ne 0$, then $t \ne 0$ and the vectors $t$ and $\xi$ are not parallel.
Therefore, $R_b\frac{\left[ W^2|(k_1,k_2)|\cos \gamma+w|t|\cos \theta\right] ^2}{1+|t|^2\sin^2\theta}$ is a rational function of $|t|$ whose numerator is of degree less than or equal to $6$, and denominator is of degree $6$. Therefore, it is a bounded function when $|t|$ (or equivalently W) tends to infinity. Hence, we prove  the inequality \eqref{eqn4.39}. And this proves the theorem.
\end{proof}
\par Now the theorem proved by L. Simon (Theorem 4.1 of \cite{LS2}) and from Theorem \ref{thm3} we conclude that
     \begin{theorem}
A minimal surface in a Matsumoto space $(V^3,F_b)$, which is a graph of a function defined on $\mathbb{R}^2$, is a plane.
     \end{theorem}
 \section{The characterization of minimal surfaces of translation surfaces}
  In this section we study the  minimal translation surface $M^2$ in Matsumoto space $(V^3,F_b)$, where $V^3$ is a real vectoe space and $F_b=\frac{{\tilde{\alpha}}^2}{\tilde{\alpha}-\tilde{\beta}}$ is a Matsumoto metric, where $\tilde{\alpha}$ is the Euclidean metric and $\tilde{\beta}=bdx^3$ is a one-form. Here we consider the immersion $\varphi : U \subset \mathbb{R}^2 \to (V^3, F_b)$ given by $\varphi(x^1, x^2) =(x^1, x^2, f(x^1)+g(x^2))$. At first we show that the pullback metric of $F_b$ by $\varphi$ is again a Matsumoto metric and then find the characterization equation the surface to be minimal.
  \begin{proposition}
  Let $F_b=F_b=\frac{{\tilde{\alpha}}^2}{\tilde{\alpha}-\tilde{\beta}}$ is a Matsumoto metric, where $\tilde{\alpha}$ is the Euclidean metric and $\tilde{\beta}=bdx^3$ is a one-form on the real vector space $V^3$. Now suppose $\varphi : U \subset \mathbb{R}^2 \to (V^3, F_b)$ given by $\varphi(x^1, x^2) =(x^1, x^2, f(x^1)+g (x^2))$,where $f$ and $g$ is a real valued smooth function be an immersion then the pullback metric on $U$ defined by \eqref{eqn2.1} is again a Matsumoto metric.
  \end{proposition}
  \begin{proof}
  We have,
  \begin{equation}
  F_b=F_b=\frac{{\tilde{\alpha}}^2}{\tilde{\alpha}-\tilde{\beta}}=\frac{(d\tilde{x}^1)^2+(d\tilde{x}^2)^2+(d\tilde{x}^3)^2}{\sqrt{(d\tilde{x}^1)^2+(d\tilde{x}^2)^2+(d\tilde{x}^3)^2}-bd\tilde{x}^3}
  \end{equation}
  Now,
  \begin{equation}
  \begin{split}
  \varphi^*(d\tilde{x}^1)=dx_1, \quad \varphi^*(d\tilde{x}^2)=dx_2, \\ \varphi^*(d\tilde{x}^3)=(d(f(x^1)+g (x^2)))=f_{x^1}dx^1+g_{x^2}dx^2
  \end{split}
  \end{equation}
  Therefore,
  \begin{equation}
  \varphi^*F_b=\frac{(1+f^2_{x^1})(dx^1)^2+2f_{x^1}g_{x^2}dx^1dx^2+(1+g^2_{x^2})(dx^2)^2}{\sqrt{(1+f^2_{x^1})(dx^1)^2+2f_{x^1}g_{x^2}dx^1dx^2+(1+g^2_{x^2})(dx^2)^2}-b(f_{x^1}dx^1+g_{x^2}dx^2)}
  \end{equation}
  which is a Matsumoto metric of the form $\frac{\alpha^2}{\alpha-\beta}$, where
  \begin{equation*}
  \alpha^2=(1+f^2_{x^1})(dx^1)^2+2f_{x^1}g_{x^2}dx^1dx^2+(1+g^2_{x^2})(dx^2)^2
  \end{equation*}
  is a Riemannian metric and
  \begin{equation*}
  \beta=b(f_{x^1}dx^1+g_{x^2}dx^2)
  \end{equation*}
  is a one-form.
  \end{proof}
  \par Let us consider the following immersion:
 \begin{equation}
 \varphi(x^1,x^2)=(\varphi^1,\varphi^2,\varphi^3)=\left( x^1,x^2,f(x^1)+g(x^2)\right) 
 \end{equation}
 Then we can write
 \begin{equation}
 \varphi^j=\delta_{j1}x_1+\delta_{j2}x_2+(f+g)\delta_{j3}, \quad  1\le \epsilon \le 3.
 \end{equation}
 Therefore, we get 
  \begin{equation}\label{eqn5.1}
  A = 
  \begin{pmatrix}
  1+f^2_{x^1} & f_{x^1}g_{x^2} \\
  f_{x^1}g_{x^2} & 1+g^2_{x^2} \\
  \end{pmatrix},
  \end{equation}
  \begin{equation}\label{eqn5.2}
  C=\sqrt{det A}=\sqrt{1+f^2_{x^1}+g^2_{x^2}} \quad \textnormal{and} \quad E=b^2(f^2_{x^1}+g^2_{x^2}).
  \end{equation}
  Here we choose $v=\varphi_{x^1}\times \varphi_{x^2} $. Then $v=(v^1,v^2,v^3)=(-f_{x^1},-g_{x^2},1)$. Hence, $v^i=-\delta_{i3}f_{x^1}-\delta_{i3}g_{x^2}+\delta_{i3}, \quad 1\le i\le 3$.\\
 By some simple calculations we can have
 \begin{equation}\label{eqn5.3}
 \frac{\partial C}{\partial z^i_{\epsilon}}v^i=0,
 \end{equation}
 \begin{equation}\label{eqn5.4}
 \frac{\partial E}{\partial z^i_{\epsilon}}v^i= 2b^2(\delta_{\epsilon 1}f_{x^1}+\delta_{\epsilon 2}g_{x^2}),
 \end{equation}
 \begin{equation}\label{eqn5.5}
 \frac{\partial C}{\partial z^j_{\eta}}\frac{\partial^2\varphi^j} {\partial x^{\epsilon}\partial x^{\eta}}= \frac{\delta_{\epsilon 1}f_{x^1}f_{x^1x^1}+\delta_{\epsilon 2}g_{x^2x^2}}{C},
 \end{equation}
 \begin{equation}\label{eqn5.6}
 \frac{\partial E}{\partial z^j_{\eta}}\frac{\partial^2\varphi^j} {\partial x^{\epsilon}\partial x^{\eta}}= 2b^2(\delta_{\epsilon 1}f_{x^1}f_{x^1x^1}+\delta_{\epsilon 2}g_{x^2x^2}),
 \end{equation}
 \begin{equation}\label{eqn5.7}
 \frac{\partial^2 E}{\partial z^i_{\epsilon}\partial z^j_{\eta}}\frac{\partial^2\varphi^j} {\partial x^{\epsilon}\partial x^{\eta}}v^i= 2b^2\left[ (1+g^2_{x^2x^2})f_{x^1x^1}+(1+f^2_{x^1x^1})g_{x^2x^2}\right] ,
 \end{equation}
 \begin{equation}\label{eqn5.8}
 \frac{\partial^2 C^2}{\partial z^i_{\epsilon}\partial z^j_{\eta}}\frac{\partial^2\varphi^j} {\partial x_{\epsilon}\partial x_{\eta}}v^i= 2\left[ (1+g^2_{x^2x^2})f_{x^1x^1}+(1+f^2_{x^1x^1})g_{x^2x^2}\right],
 \end{equation}
  Using \eqref{eqn5.3} in \eqref{MCE0} we have
 \begin{eqnarray}\label{MCE002}
 \frac{\partial^2 \varphi^j}{\partial x^{\epsilon}\partial x^{\eta}}v^i\left[ \frac{\partial^2 C^2}{\partial z^i_{\epsilon}\partial z^j_{\eta}}(2C^2+3E)(2C^2+E)-\frac{\partial^2 E}{\partial z^i_{\epsilon}\partial z^j_{\eta}} 2C^2(2C^2+E)\nonumber \right.\\ \left.+\left\lbrace\frac{\partial E }{\partial z^i_\epsilon}\frac{\partial C }{\partial z^j_\eta} \left(4C^3-6CE \right)+4C^2 \frac{\partial E }{\partial z^i_\epsilon}\frac{\partial E }{\partial z^j_\eta} \right\rbrace  \right]&=&0. \hspace{1.0cm}
 \end{eqnarray}
 Therefore, using \eqref{eqn5.2} to \eqref{eqn5.8} in \eqref{MCE002} we obtain 
 \begin{equation}\label{eqn5.11}
 \begin{split}
f_{x^1x^1}(1+g^2_{x^2})\left[2+(2+b^2)(f^2_{x^1}+g^2_{x^2})\right]\left[ 2(1-b^2)(2+b^2)(f^2_{x^1}+g^2_{x^2})\right] \\ +g_{x^2x^2}(1+f^2_{x^1})\left[2+(2+b^2)(f^2_{x^1}+g^2_{x^2})\right]\left[ 2(1-b^2)(2+b^2)(f^2_{x^1}+g^2_{x^2})\right]=0 .
 \end{split}
 \end{equation}
Hence, we have the following theorem,
\begin{theorem}
Let $\varphi:M^2 \to (V^3,F_b)$ be an immersion in a Matsumoto space with local coordinates $(\varphi^i(x))$. Then $\varphi$ is minimal if and only if 
\begin{equation}\label{eqn5.12}
\lambda f_{x^1x^1}+ \mu g_{x^2x^2}=0
\end{equation}
where,
\begin{equation}\label{eqn5.13}
\begin{split}
\lambda=(1+g^2_{x^2})\left[2+(2+b^2)(f^2_{x^1}+g^2_{x^2})\right]\left[ 2(1-b^2)(2+b^2)(f^2_{x^1}+g^2_{x^2})\right]\\
+6b^2f^2_{x^1}\left\lbrace 2+(2-b^2) (f^2_{x^1}+g^2_{x^2})\right\rbrace 
\end{split}
\end{equation}
and
\begin{equation}\label{eqn5.14}
\begin{split}
\mu= (1+f^2_{x^1})\left[2+(2+b^2)(f^2_{x^1}+g^2_{x^2})\right]\left[ 2(1-b^2)(2+b^2)(f^2_{x^1}+g^2_{x^2})\right]\\
+6b^2g^2_{x^2}\left\lbrace 2+(2-b^2) (f^2_{x^1}+g^2_{x^2})\right\rbrace
\end{split}
\end{equation}
\end{theorem}
Now we want to solve the differential equation \eqref{eqn5.11}. Let $r=f^2_{x^1}$ and $s=g^2_{x^2}$. Then
\begin{equation}
f_{x^1x^1}=\frac{r_f}{2}, \qquad g_{x^2x^2}=\frac{s_g}{2}.
\end{equation}
Then from \eqref{eqn5.13} and \eqref{eqn5.14} becomes

\begin{equation}\label{eqn5.15}
\begin{split}
\lambda=(1+s)\left[2+(2+b^2)(r+s)\right]\left[ 2(1-b^2)(2+b^2)(r+s)\right]\\
+6b^2r\left\lbrace 2+(2-b^2) (r+s)\right\rbrace 
\end{split}
\end{equation}
and
\begin{equation}\label{eqn5.16}
\begin{split}
\mu=(1+r)\left[2+(2+b^2)(r+s)\right]\left[ 2(1-b^2)(2+b^2)(r+s)\right]\\
+6b^2s\left\lbrace 2+(2-b^2) (r+s)\right\rbrace
\end{split}
\end{equation}
And \eqref{eqn5.12} becomes
\begin{equation}
r_f\lambda+s_q\mu=0
\end{equation}
Therefore, we have two cases:\\
\textbf{Case 1:} If $r_f=0$ or, $s_g=0$, then $r$ and $s$ are constant functions. And hence $f$ and $g$ are linear functions. Therefore, $M^2$ is a piece of plane in $(V^3, F_b)$.\\
\textbf{Case 2:} Let $r_f\ne 0$ and $s_g \ne 0$. Then we have, $\lambda\ne 0$ and $\mu \ne 0$. Suppose 
\begin{equation}
\kappa =\frac{r_f}{\mu}=-\frac{s_g}{\lambda}.
\end{equation}
Which implies that
\begin{equation*}
(r_f)_g=\mu_g\kappa+\mu\kappa_g=0 \quad \textnormal{and} \quad (s_g)_f=\lambda_f\kappa+\lambda\kappa_f=0
\end{equation*}
Hence, we have, 
\begin{equation}
\log{\kappa}_f= \frac{\kappa_f}{\kappa}=-\frac{\lambda_f}{\lambda}\quad \textnormal{and} \quad \log{\kappa}_g=\frac{\kappa_g}{\kappa}=-\frac{\mu_g}{\mu}.
\end{equation}
Since, $(\log{\kappa}_f)_g=(\log{\kappa}_g)_f$, we have,
\begin{equation}\label{eqn5.17}
\left(\frac{\lambda_f}{\lambda} \right)_g =\left( \frac{\mu_g}{\mu}\right)_f.
\end{equation}
We can easily observe that, $r_g=(r_f)_g=0$ and $s_f=(s_g)_f=0$. Therefore, we have, 
\begin{equation}\label{eqn5.18}
\left(\frac{\lambda_f}{\lambda} \right)_g =\left( \frac{\lambda_rr_f}{\lambda}\right) _g=\left( \frac{\lambda_r}{\lambda}\right) _gr_f=\left( \frac{\lambda_r}{\lambda}\right) _sr_fs_g
\end{equation} 
and
\begin{equation}\label{eqn5.19}
\left(\frac{\mu_g}{\mu} \right)_f =\left( \frac{\mu_ss_g}{\mu}\right) _f=\left( \frac{\mu_s}{\mu}\right) _fs_g=\left( \frac{\mu_s}{\mu}\right) _rr_fs_g
\end{equation} 
Using \eqref{eqn5.18} and \eqref{eqn5.19} in \eqref{eqn5.17} we get,
\begin{equation}
\left( \frac{\lambda_r}{\lambda}\right) _s =\left( \frac{\mu_s}{\mu}\right) _r.
\end{equation}
That is,
\begin{equation}\label{eqn5.42}
\left( \log \frac{\lambda}{\mu}\right)_{rs}=0 
\end{equation}
Let $p=r+s$ and $q=r-s$. Then we have
\begin{equation}
\lambda=K(p)-L(p)q, \qquad \mu=K(p)+L(p)q
\end{equation}
where,
\begin{equation}\label{eqn5.201}
K(p)=4(1-b^2)+\frac{p}{2}(20++8b^2-4b^4)+\frac{p^2}{2}(16+20b^2-6b^4)+\frac{p^3}{2}(2+b^2)^2
\end{equation}
\begin{equation}\label{eqn5.202}
L(p)=2(1-4b^2)+\frac{p}{2}(8-12b^2+4b^4)+\frac{p^2}{2}(2+b^2)^2
\end{equation}
Now from \eqref{eqn5.42} it follows that
\begin{equation}\label{eqn5.20}
\left( \log \frac{\lambda}{\mu}\right)_{rs}=\left( \log \frac{\lambda}{\mu}\right)_{pp}-\left( \log \frac{\lambda}{\mu}\right)_{qq}=0
\end{equation}
Now substitute the values of $\lambda$ and $\mu$ in \eqref{eqn5.20} we get
\begin{equation}
\begin{split}
q^3\left(K_{pp}L^3-KL^2L_{pp}-2K_pL_pL^2+2KLL^2_p \right) \\ +q\left( -K_{pp}K^2L+K^3L_{pp}-2K_pK^2L_p+2K^2_pKL-2KL^3\right) =0.
\end{split}
\end{equation}
Since, $t$ is an arbitrary function we get,
\begin{equation}\label{eqn5.21}
K_{pp}L^3-KL^2L_{pp}-2K_sL_pL^2+2KLL^2_p=0
\end{equation}
\begin{equation}\label{eqn5.22}
 -K_{pp}K^2L+K^3L_{pp}-2K_pK^2L_p+2K^2_pKL-2KL^3=0
\end{equation}
From \eqref{eqn5.201} and \eqref{eqn5.202} we can obtain easily that
\begin{equation}\label{eqn5.23}
\left[ \left(\frac{K}{L} \right)_p \right]^2=1. 
\end{equation}
Therefore,
\begin{equation}\label{eqn5.24}
\frac{K}{L}=p+\frac{8+32b^2-10b^4}{(2+b^2)^2}+\frac{4b^4}{T}\left(\frac{132-60b^2+9b^4}{(2+b^2)^2}p+\frac{2(66-21b^2)}{(2+b^2)^2} \right) 
\end{equation}
where, 
\begin{equation*}
T=(4-16b^2)+p(8-12b^2+4b^4)+p^2(2+b^2)^2
\end{equation*}
Now differentiating \eqref{eqn5.24} with respect to $p$ we get 
\begin{equation}\label{eqn5.25}
\left( \frac{K}{L}\right)_p=1- \frac{4b^4}{T'}\left(\frac{9b^4-102b^2+264}{(b^2+2)^2}+(94b^4-12b^2+8+2(b^2+2)^2p)\right) 
\end{equation}
where, $T'=(-16b^2+4+(94b^4-12b^2+8)p+(b^2+2)^2p^2)^2$\\
Now \eqref{eqn5.23} will true if and only if $b=0$. Hence, we obtain the following theorem:
\begin{theorem}
A minimal surface in a Matsumoto space $(V^3,F_b)$, which is the translation surface defined on $\mathbb{R}^2$, is a plane.
\end{theorem}

\end{document}